\documentclass{article}

\usepackage[utf8]{inputenc}
\usepackage[T1]{fontenc}
\usepackage{lmodern}
\usepackage{amsmath}
\usepackage{amssymb}
\usepackage{amsthm}
\usepackage{amsfonts}
\usepackage{dsfont}
\usepackage{mathtools}
\usepackage{csquotes}
\usepackage{hyperref}
\hypersetup{
  colorlinks=false,
  pdfborder={0 0 0},
  pdftitle={Analysis of a chemotaxis model with indirect signal absorption},
  pdfauthor={Mario Fuest},
  pdfkeywords={chemotaxis, indirect consumption, global existence, large-time behavior},
  bookmarksopen=true,
}

\RequirePackage{geometry}
\newcommand{\R}{\mathbb{R}}

\newcommand{\N}{\mathbb{N}}

\newcommand{\ur}[1]{\mathrm{#1}}
\newcommand{\ure}{\ur e}

\ifdefined\labelenumi
  \renewcommand{\labelenumi}{(\roman{enumi})}
  
\fi

\newcommand{\eps}{\varepsilon}

\newcommand{\gt}{>}
\newcommand{\lt}{<}

\newcommand{\defs}{\coloneqq}
\newcommand{\sfed}{\eqqcolon}

\newcommand{\ra}{\rightarrow}

\newcommand{\nea}{\nearrow}

\newcommand{\sea}{\searrow}

\newcommand{\ol}{\overline}
\newcommand{\ul}{\underline}

\newcommand{\diff}{\,\mathrm{d}}

\newcommand{\ds}{\,\mathrm{d}s}
\newcommand{\dt}{\,\mathrm{d}t}

\newcommand{\ddt}{\frac{\mathrm{d}}{\mathrm{d}t}}

\DeclareMathOperator{\sign}{sign}
\newcommand{\loc}{\mathrm{loc}}

\newcommand{\embed}{\hookrightarrow}

\newcommand{\hp}{\hphantom}
\newcommand{\pe}{\mathrel{\hp{=}}}

\newcommand{\tmax}{T_{\max}}
\newcommand{\intom}{\int_\Omega}
\newcommand{\ombar}{\ol \Omega}

\let\originalparagraph\paragraph
\renewcommand{\paragraph}[2][.]{\originalparagraph{#2#1}}

\newtheoremstyle{nplain}
  {\topsep}   
  {\topsep}   
  {\itshape}  
  {0pt}       
  {\bfseries} 
  {.}         
  {5pt plus 1pt minus 1pt} 
  {\thmnumber{#2 }\thmname{#1}\thmnote{ (#3)}} 

\newtheoremstyle{ndefinition}
  {\topsep}   
  {\topsep}   
  {}   	      
  {0pt}       
  {\bfseries} 
  {.}         
  {5pt plus 1pt minus 1pt} 
  {\thmnumber{#2 }\thmname{#1}\thmnote{ (#3)}} 

\newtheorem{base}{Base}[section]
\numberwithin{equation}{section}

\theoremstyle{nplain}
\newtheorem{theorem}[base]{Theorem} \newtheorem*{theorem*}{Theroem}
\newtheorem{lemma}[base]{Lemma} \newtheorem*{lemma*}{Lemma}
\newtheorem{prop}[base]{Proposition} \newtheorem*{prop*}{Proposition}
 \newtheorem*{cor*}{Corollary}

\theoremstyle{ndefinition}
 \newtheorem*{definition*}{Definition}
 \newtheorem*{example*}{Example}
\newtheorem{cond}[base]{Condition} \newtheorem*{cond*}{Condition}

\newtheorem{remark}[base]{Remark} \newtheorem*{remark*}{Remark}

\begin{document}
\title{Analysis of a chemotaxis model with indirect signal absorption}
\author{
Mario Fuest\footnote{fuestm@math.uni-paderborn.de}\\
{\small Institut f\"ur Mathematik, Universit\"at Paderborn,}\\
{\small 33098 Paderborn, Germany}
}
\date{}
\maketitle

\begin{abstract}
\noindent
We consider the chemotaxis model
\begin{align*}
  \begin{cases}
    u_t = \Delta u - \nabla \cdot (u \nabla v), \\
    v_t = \Delta v - vw, \\
    w_t = -\delta w + u \\
  \end{cases}
\end{align*}
in smooth, bounded domains $\Omega \subset \R^n$, $n \in \N$,
where $\delta \gt 0$ is a given parameter.\\[5pt]
If either $n \le 2$ or $\|v_0\|_{L^\infty(\Omega)} \le \frac1{3n}$
we show the existence of a unique global classical solution $(u, v, w)$
and convergence of $(u(\cdot, t), v(\cdot, t), w(\cdot, t))$ towards a spatially constant equilibrium, as $t \ra \infty$.
\\[5pt]
The proof of global existence for the case $n \le 2$ relies on a bootstrap procedure.
As a starting point we derive a functional inequality for a functional being sublinear in $u$,
which appears to be novel in this context.
\\[5pt]
 \textbf{AMS Classification (2010):} {35K55 (primary), 35A01, 35K40, 92C17 (secondary)}
\end{abstract}

\section{Introduction} \label{sec:intro}
\paragraph{The model}
Organisms such as cells or bacteria may partially direct their movement towards an external chemical signal.
This process is known as chemotaxis and corresponding mathematical models
have been introduced by Keller and Segel \cite{KellerSegelTravelingBandsChemotactic1971} in the 1970s.
The most prototypical system is
\begin{align}
  \begin{cases} \label{prob:keller_segel}
    u_t = \Delta u - \nabla \cdot (u \nabla v), \\
    v_t = \Delta v - v + u,
  \end{cases}
\end{align}
wherein $u$ and $v$ denote the cell/bacteria density and the concentration of the chemical signal, respectively.
Its most striking feature is the possibility of chemotactic collapse;
that is, the existence of solutions in space-dimension two
\cite{HorstmannWangBlowupChemotaxisModel2001, SenbaSuzukiParabolicSystemChemotaxis2001}
and higher \cite{WinklerFinitetimeBlowupHigherdimensional2013}
blowing up in finite time.
In the past few decades mathematicians have analyzed several chemotaxis models;
for a broader introduction we refer to the survey \cite{BellomoEtAlMathematicalTheoryKeller2015}.

However, even simpler organisms may orient their movement towards a nutrient which is consumed rather than produced,
leading to the model
\begin{align}
  \begin{cases} \label{prob:oxy_consumption}
    u_t = \Delta u - \nabla \cdot (u \nabla v), \\
    v_t = \Delta v - uv.
  \end{cases}
\end{align}
In space-dimensions one and two
for any sufficiently smooth initial data classical solutions to \eqref{prob:oxy_consumption} exist globally
and converge to steady states
\cite{ZhangLiStabilizationConvergenceRate2015},
while in space-dimension three at least weak solutions have been constructed which become eventually smooth
\cite{TaoWinklerEventualSmoothnessStabilization2012}.

For higher space-dimensions $n$ globality of classical solutions
has been shown for sufficiently small values of $\|v_0\|_{L^\infty(\Omega)}$.
Tao \cite{TaoBoundednessChemotaxisModel2011} proved that whenever
the corresponding initial data are sufficiently smooth and satisfy $\|v_0\|_{L^\infty(\Omega)} \le \frac{1}{6(n+1)}$,
then there exists a global classical solution of \eqref{prob:oxy_consumption}.
In \cite{BaghaeiKhelghatiBoundednessClassicalSolutions2017} this condition has then been improved;
it is sufficient to require $\|v_0\|_{L^\infty(\Omega)} \lt \frac{\pi}{\sqrt{2(n+1)}}$.

In addition, chemotaxis-consumptions models have been embedded into more complex frameworks.
For instance,
coupled chemotaxis-fluid systems
\cite{LorzCoupledChemotaxisFluid2010,
WinklerGlobalLargeDataSolutions2012,
WinklerGlobalWeakSolutions2016,
WinklerStabilizationTwodimensionalChemotaxisNavier2014},
systems with nonlinear diffusion and/or nonlinear chemotactic sensitivity
\cite{FanJinGlobalExistenceAsymptotic2017,
LankeitLocallyBoundedGlobal2017,
LiuDongGlobalSolutionsQuasilinear2018,
ZhaoZhengAsymptoticBehaviorChemotaxis2018}
or systems with zeroth order terms accounting for
logistic growth \cite{LankeitWangGlobalExistenceBoundedness2017}
or competition between species \cite{WangEtAlBoundednessAsymptoticStability2018} have been analyzed.

However, models accounting for indirect consumption effects have apparently not been treated in mathematical literature yet. 
This stands in contrast to the case of signal production,
where indirect effects have been studied for example in
\cite{FujieSenbaApplicationAdamsType2017, QiuEtAlBoundednessHigherdimensionalQuasilinear2018, TaoWinklerCriticalMassInfinitetime2017}.

In the present work, we analyze a prototypical chemotaxis system with indirect consumption; that is, we study
\begin{align*} \label{prob:p} \tag{\text{P}}
  \begin{cases}
    u_t = \Delta u - \nabla \cdot (u \nabla v),              & \text{in $\Omega \times (0, T)$}, \\
    v_t = \Delta v - vw,                                     & \text{in $\Omega \times (0, T)$}, \\
    w_t = -\delta w + u,                                     & \text{in $\Omega \times (0, T)$}, \\
    \partial_\nu u = \partial_\nu v = 0,                     & \text{on $\partial \Omega \times (0, T)$}, \\
    u(\cdot, 0) = u_0, v(\cdot, 0) = v_0, w(\cdot, 0) = w_0, & \text{in $\Omega$}
  \end{cases}
\end{align*}
for $T \in (0, \infty]$, a smooth, bounded domain $\Omega \subset \R^n$, $n \in \N$, a parameter $\delta \gt 0$
and given initial data $u_0, v_0, w_0$.

\paragraph{Main ideas and results I: Global existence}
We start by stating a local existence result in Lemma~\ref{lm:local_ex}
which already gives a criterion for global existence.
In the following we improve the condition,
it suffices to show an $L^p$ bound for $u$ for sufficiently high $p$ (cf.\ Proposition~\ref{prop:global_ex_n}).
We will then proceed to gain such bounds.

At first glance, one might suspect that chemotaxis-consumption models such as \eqref{prob:oxy_consumption} or \eqref{prob:p}
are easier to handle than chemotaxis-production models such as \eqref{prob:keller_segel}.
After all, the comparison principle rapidly warrants
that $0 \le v \le \|v_0\|_{L^\infty(\Omega)}$ (cf.\ Lemma~\ref{lm:positivity_mass_const} below).
While indeed helpful, an $L^\infty$-bound for $v$ does not immediately solve all problems,
since such a bound does not directly imply any bounds of $\nabla v$,
the term appearing in the first equation of \eqref{prob:oxy_consumption} and \eqref{prob:p}.
In addition, an important tool for analyzing \eqref{prob:keller_segel} and variants thereof is to prove a certain functional inequality
which simply does not seem to be available for chemotaxis-consumption models.

In many cases, for instance in
\cite{LankeitWangGlobalExistenceBoundedness2017, TaoWinklerEventualSmoothnessStabilization2012, TaoWinklerLocallyBoundedGlobal2013},
the authors utilize the functional 
\begin{align} \label{eq:not_working_functional}
  \intom u \log u + 2 \intom |\nabla \sqrt{v}|^2
\end{align}
to handle problems similar to \eqref{prob:oxy_consumption}.
The \enquote{worst} term appearing upon derivating $\intom u \log u$ is $\intom \nabla u \cdot \nabla v$,
while upon derivating $\intom |\nabla \sqrt{v}|^2$ the term $-\frac12 \intom \nabla u \cdot \nabla v$ shows up.
Hence by calculating the derivative of \eqref{eq:not_working_functional} these terms cancel out each other.

However, if we tried to follow this approach for the system \eqref{prob:p} we would obtain
\begin{align*}
  \intom \nabla u \cdot \nabla v - \intom \nabla v \cdot \nabla w
\end{align*}
instead.
Even ignoring the fact that $w$ might not be smooth enough to justify the calculation,
it is not clear at all how to handle these terms.

Therefore it seems necessary to follow a different approach.
In order to prove global existence, we will rely on functionals of the form
\begin{align} \label{eq:working_functional}
  \intom u^p \varphi(v)
\end{align}
for certain functions $\varphi$ and $p \gt 0$ (cf.\ Lemma~\ref{lm:u_p_varphi_v_estimate}).

For instance in \cite{TaoBoundednessChemotaxisModel2011, WinklerAbsenceCollapseParabolic2010}
such functionals have been capitalized for $p \gt 1$.
Indeed, for sufficiently small $v_0$ such an approach leads to success also for \eqref{prob:p},
see Proposition~\ref{prop:global_ex_small_v0}.

Functionals of the form of \eqref{eq:working_functional} have also already been studied with $p \in (0, 1)$ in
\cite{LankeitWinklerGeneralizedSolutionConcept2017, StinnerWinklerGlobalWeakSolutions2011},
in both cases with $\varphi(s) = s^q$, $s \ge 0$, for some $q \gt 0$.
However, in those works they have only helped to obtain weak solutions:
The general idea is to obtain space-time bounds of expressions such as $|\nabla u^{\frac{p}{2}}|^2 \varphi(v)$;
that is, one might then hope to construct (global in time) solutions $(u_\eps, v_\eps)$, $\eps \gt 0$, to approximate problems
and derive space-time-bounds of, for instance, $|\nabla u_\eps^\frac{p}{2}|^2 \varphi(v_\eps)$ independently of $\eps$,
allowing for the application of certain convergence theorems.

However, such bounds seemingly cannot be utilized to obtain global classical solutions.
Here lies the crucial difference in the present problem; the special structure of \eqref{prob:p} allows us to go further:
In the quite simple but essential Lemma~\ref{lm:spacetime_u_space_w} we prove
that space-time bounds for $u$ imply uniform-in-time space bounds for $w$.
This allows us (at least in space-dimension one and two)
to undertake a bootstrap procedure in Proposition~\ref{prop:global_ex_n2}:
These bounds imply bounds for $v$ in certain Sobolev spaces,
which then imply improved space-time bounds for $u$, which again provide space estimates for $w$ and so on.

Finally, we are able to prove
\begin{theorem} \label{th:global_ex}
  Let $\Omega \subset \R^n$, $n \in \N$, be a bounded, smooth domain and $\beta \in (0, 1)$.
  Suppose that
  \begin{align} \label{eq:main:initial_data}
    u_0 \in C^0(\ombar), \quad
    v_0 \in W^{1, \infty}(\Omega) \quad \text{and} \quad
    w_0 \in C^\beta(\ombar)
  \end{align}
  satisfy
  \begin{align} \label{eq:main:initial_nonnegative}
    u_0, v_0, w_0 \ge 0 \text{ in $\ombar$}
    \quad \text{and} \quad
    u_0 \not\equiv 0,
  \end{align}
  and if $n \ge 3$ also
  \begin{align*}
    \|v_0\|_{L^\infty(\Omega)} \le \frac{1}{3n}.
  \end{align*}
  Then there exists a global classical solution $(u, v, w)$ of problem~\eqref{prob:p}
  which is uniquely determined by the inclusions
  \begin{align*} 
    u &\in C^0(\ombar \times [0, \infty)) \cap C^{2, 1}(\ombar \times (0, \infty)), \\
    v &\in \bigcap_{\theta \gt n} C^0([0, \infty); W^{1, \theta}(\Omega)) \cap C^{2, 1}(\ombar \times (0, \infty)) \\
  \intertext{and}
    w &\in C^0(\ombar \times [0, \infty)) \cap C^{0, 1}(\ombar \times (0, \infty)).
  \end{align*}
\end{theorem}

\paragraph{Main ideas and results II: Large time behavior}
Having obtained global solutions we examine their large time behavior in Section~\ref{sec:large_time}.

The main challenge lies in the fact that the aforementioned bootstrap procedure for the case $n \le 2$
only implies local-in-time boundedness of the solution components.
Therefore we revise our arguments of Section~\ref{sec:global_existence} to show
that $\nabla v$ is uniformly in time bounded in $L^\theta(\Omega)$ for some $\theta \gt n$, see Proposition~\ref{prop:cond_fulfilled}.

Along with a very weak convergence result (Lemma~\ref{lm:very_weak_conv_v})
this allows us to deduce $v(\cdot, t) \ra 0$ as $t \ra \infty$ in $L^\infty(\Omega)$, see Lemma~\ref{lm:convergence_v}.

The results of Section~\ref{sec:global_existence} then allow us to find $t_0 \gt 0$
such that the solution to \eqref{prob:p} with initial data $u(\cdot, t_0), v(\cdot, t_0), w(\cdot, t_0)$
is bounded in $L^\infty(\Omega) \times W^{1, \infty}(\Omega) \times L^\infty(\Omega)$.
Due to uniqueness this implies certain bounds for $u, v$ and $w$ as well.
By using parabolic regularity theory we then improve this to bounds in certain Hölder spaces (Lemma~\ref{lm:uvw_bounded_hoelder}).

Since we are also able to deduce a very weak convergence result for $u$ in Lemma~\ref{lm:very_weak_conv_u},
we may use this regularity result in order to obtain convergence of $u$ (Lemma~\ref{lm:convergence_u}) --
which in turn together with the variations of constants formula implies convergence of $w$ (Lemma~\ref{lm:convergence_w}).

In the end, we arrive at
\begin{theorem} \label{th:conv}
  Under the assumptions of Theorem~\ref{th:global_ex} there exists $\alpha \in (0, 1)$
  such that the solution $(u, v, w)$ given by Theorem~\ref{th:global_ex} fulfills 
  \begin{align} \label{eq:main:function_spaces}
    u, v \in C_{\loc}^{2+\alpha, 1+\frac{\alpha}{2}}(\ombar \times [1, \infty))
    \quad \text{and} \quad 
    w    \in C_{\loc}^{\alpha, 1+\frac{\alpha}{2}}(\ombar \times [1, \infty))
  \end{align}
  as well as
  \begin{align} \label{eq:main:convergence}
    u(\cdot, t) \ra \ol u_0 \text{ in $C^{2+\alpha}(\ombar)$}, \quad
    v(\cdot, t) \ra 0 \text{ in $C^{2+\alpha}(\ombar)$} \quad \text{and} \quad
    w(\cdot, t) \ra \frac{\ol u_0}{\delta} \text{ in $C^{\alpha}(\ombar)$},
    \qquad \text{as $t \ra \infty$},
  \end{align}
  wherein
  \begin{align*}
    \ol u_0 \defs \frac{1}{|\Omega|} \intom u_0.
  \end{align*}
\end{theorem}

\section{Preliminaries} \label{sec:prelims}
Henceforth we fix a smooth, bounded domain $\Omega \subset \R^n$, $n \in \N$.

We start by stating a local existence result.
\begin{lemma} \label{lm:local_ex}
  Suppose $u_0, v_0, w_0: \ombar \ra \R$ satisfy \eqref{eq:main:initial_data} for some $\beta \in (0, 1)$.
  Then there exist $\tmax \in (0, \infty)$ and functions
  \begin{align} 
    u &\in C^0(\ombar \times [0, \tmax)) \cap C^{2, 1}(\ombar \times (0, \tmax)), \label{eq:local_ex:reg_u} \\
    v &\in \bigcap_{\theta \gt n} C^0([0, \tmax); W^{1, \theta}(\Omega)) \cap C^{2, 1}(\ombar \times (0, \tmax)) \label{eq:local_ex:reg_v}
  \intertext{and}
    w &\in C^0(\ombar \times [0, \tmax)) \cap C^{0, 1}(\ombar \times (0, \tmax)) \label{eq:local_ex:reg_w}
  \end{align}
  solving \eqref{prob:p} classically and are such that if $\tmax \lt \infty$, then 
  \begin{align} \label{eq:local_ex:ex_crit}
    \limsup_{t \nea \tmax} \left( \|u(\cdot, t)\|_{L^\infty(\Omega)} + \|\nabla v(\cdot, t)\|_{L^\theta(\Omega)} \right) = \infty
  \end{align}
  for all $\theta \gt n$.
  These functions are uniquely determined
  by the inclusions \eqref{eq:local_ex:reg_u}, \eqref{eq:local_ex:reg_v} and \eqref{eq:local_ex:reg_w}
  and can be represented by
  \begin{align}
        u(\cdot, t)
    &= \ure^{t \Delta} u_0  - \int_0^t \ure^{(t-s) \Delta} \nabla \cdot (u(\cdot, s) \nabla v(\cdot, s)) \ds, \label{eq:tilde_repr_u} \\
        v(\cdot, t)
    &= \ure^{t \Delta} v_0  - \int_0^t \ure^{(t-s) \Delta} \left( v(\cdot, s) w(\cdot, s) \right) \ds, \label{eq:tilde_repr_v} \\
  \intertext{and}
        w(\cdot, t)
    &= \ure^{-\delta t} w_0 + \int_0^t \ure^{-\delta(t-s)} u(\cdot, s) \ds \label{eq:tilde_repr_w}
  \end{align}
  for $t \in (0, \tmax)$.
\end{lemma}
\begin{proof}
  This can be shown by a fixed point argument as (inter alia) in \cite[Theorem~3.1]{HorstmannWinklerBoundednessVsBlowup2005}.
  Let us briefly recall the main idea:
  Let $\theta \gt n$ be arbitrary.
  For sufficiently small $T \gt 0$ the map $\Phi$ given by
  \begin{align*}
      \Phi(u, v)
     =\begin{pmatrix}
        t \mapsto \ure^{t \Delta} u_0 - \int_0^t \ure^{(t-s) \Delta} \nabla \cdot (u(\cdot, s) \nabla v(\cdot, s)) \ds \\
        t \mapsto \ure^{t \Delta} v_0 - \int_0^t \ure^{(t-s) \Delta} v(\cdot, s) (\Psi(u))(\cdot, s) \ds
      \end{pmatrix}
  \end{align*}
  with
  \begin{align*}
    \Psi(u)): t \mapsto \ure^{-\delta t} w_0 + \int_0^t \ure^{-\delta (t-\sigma)} u(\cdot, \sigma) \diff\sigma
  \end{align*}
  acts as a contraction on a certain closed subset of the Banach space
  \begin{equation*}
    C^0([0, T]; C^0(\ombar)) \times C^0([0, T]; W^{1, \theta}(\Omega)).
  \end{equation*}
  By Banach's fixed point theorem one then obtains a unique tuple $(u, v)$
  such that $(u, v, w)$ with $w \defs \Psi(u)$ satisfies 
  \eqref{eq:tilde_repr_u}, \eqref{eq:tilde_repr_v} and \eqref{eq:tilde_repr_w}
  for $t \in (0, T)$.
  Repeating this argument leads to the extensibility criterion \eqref{eq:local_ex:ex_crit}.

  In order to show that \eqref{eq:local_ex:reg_u} and \eqref{eq:local_ex:reg_v} hold,
  one uses parabolic regularity theory, similar as in for example~\cite{HorstmannWinklerBoundednessVsBlowup2005}.
  Here it is important to note that Hölder regularity of $u$ implies Hölder regularity of $w$
  (cf.\ Lemma~\ref{lm:bound_u_bound_w} below).
\end{proof}

\begin{lemma} \label{lm:positivity_mass_const}
  For any $u_0, v_0, w_0: \ombar \ra \R$ satisfying \eqref{eq:main:initial_data} for some $\beta \in (0, 1)$
  and \eqref{eq:main:initial_nonnegative}
  the solution $(u, v, w)$ constructed in Lemma~\ref{lm:local_ex} fulfills 
  \begin{equation*}
    u \gt 0, \quad
    v \ge 0, \quad
    v \le \|v_0\|_{L^\infty(\Omega)}
    \quad \text{and} \quad
    w \gt 0
  \end{equation*}
  in $\ombar \times (0, \tmax)$, where $\tmax$ is given by Lemma~\ref{lm:local_ex}.
  Furthermore, for all $t \in [0, \tmax)$ we have
  \begin{align} \label{eq:positivity_mass_const:mass_const}
    \intom u(\cdot, t) = \intom u_0 \sfed m.
  \end{align}
\end{lemma}
\begin{proof}
  By comparison we have $u \gt 0$ and $v \ge 0$ and then also $w \ge 0$,
  hence $-vw \le 0$ and therefore also by comparison $v \le \|v_0\|_{L^\infty(\Omega)}$ in $\ol \Omega \times (0, \tmax)$.

  Moreover, integrating the first equation in \eqref{prob:p} over $\Omega$
  yields \eqref{eq:positivity_mass_const:mass_const}.
\end{proof}

For the remainder of this article
we fix $u_0, v_0, w_0: \ombar \ra \R$ satisfying \eqref{eq:main:initial_data}
and \eqref{eq:main:initial_nonnegative} for some $\beta \in (0, 1)$
and let always $(u, v, w)$ and $\tmax$ be as in Lemma~\ref{lm:local_ex}.

\section{Global existence} \label{sec:global_existence}

\subsection{Enhancing the extensibility criterion}
We begin by providing some useful estimates.
\begin{lemma} \label{lm:bound_u_bound_w}
  There exists $C \gt 0$ such that for any function space
  \begin{align*}
    X \in \bigcup_{p \in [1, \infty]} L^p(\Omega) \cup \bigcup_{\alpha \in [0, \beta]} C^\alpha(\ombar)
  \end{align*}
  the inequality
  \begin{align*}
    \|w\|_{L^\infty((0, \tmax); X)} \le C(1 + \|u\|_{L^\infty((0, \tmax); X)})
  \end{align*}
  holds.
\end{lemma}
\begin{proof}
  By \eqref{eq:tilde_repr_w} we have
  \begin{align*}
          \sup_{t \in (0, \tmax)} \|w(\cdot, t)\|_X
    &\le  \sup_{t \in (0, \tmax)} \left(
            e^{-\delta t} \|w_0\|_X
            + \int_0^t e^{-\delta (t-s)} \|u(\cdot, s)\|_X \ds
          \right) \\
    &\le  \sup_{t \in (0, \tmax)} \left(
            \|w_0\|_X
            + \|u\|_{L^\infty((0, \tmax); X)} \int_0^t e^{-\delta (t-s)} \ds
          \right) \\
    &\le  \sup_{t \in (0, \tmax)} \left( \|w_0\|_X + \frac1{\delta} \|u\|_{L^\infty((0, \tmax); X)} \right),
  \end{align*}
  such that the statement follows by setting $C \defs \max\{\frac{1}{\delta}, \|w_0\|_X\}$,
  which is finite as $w_0 \in X$ by \eqref{eq:main:initial_data}.
\end{proof}

\begin{lemma} \label{lm:lp_lq_nabla_v}
  Let $p \ge 1$.
  For all 
  \begin{align*}
    \theta \in
    \begin{cases}
      [1, \tfrac{np}{n-p}), & p \le n, \\
      [1, \infty],          & p \gt n
    \end{cases}
  \end{align*}
  there exists a constant $C \gt 0$ such that
  \begin{align*}
        \|v\|_{L^\infty((0, \tmax); W^{1, \theta}(\Omega))}
   \le  C \left(1 + \|w\|_{L^\infty((0, \tmax); L^{p}(\Omega))} \right)
  \end{align*}
  holds.
\end{lemma}
\begin{proof}
  Due to Hölder's inequality we may assume without loss of generality $\theta \ge p$.

  Set
  \begin{align} \label{eq:lp_lq_nabla_v:gamma}
        \gamma
    \defs  -\frac12 - \frac{n}{2} \left( \frac{1}{p} - \frac{1}{\theta} \right).
  \end{align}
  
  Then we have for $p \le n$
  \begin{align*} 
        \gamma
    \gt -\frac12 - \frac{n}{2} \left( \frac{1}{p} - \frac{n-p}{np} \right)
    =   -\frac12 - \frac{n}{2} \cdot \frac{1}{n}
    =   -1
  \end{align*}
  and for $p \gt n$
  \begin{align*} 
        \gamma
    \ge -\frac12 - \frac{n}{2} \cdot \left(\frac{1}{p} - \frac{1}{\infty} \right)
    =   -\frac12 - \frac{n}{2p}
    \gt -1.
  \end{align*}
   
  By known smoothing estimates for the Neumann Laplace semigroup (cf.\ \cite[Lemma~1.3~(ii) and (iii)]{WinklerAggregationVsGlobal2010}
  and Hölder's inequality
  there exist $c_1, c_2, \lambda \gt 0$ such that
  \begin{align*}
          \|\nabla \ure^{\sigma \Delta} \varphi\|_{L^\theta(\Omega)}
    &\le  c_2 (1 + \sigma^\gamma) e^{-\lambda \sigma}\|\varphi\|_{L^p(\Omega)} \quad \text{for all $\varphi \in L^p(\Omega)$ and all $\sigma \gt 0$}
  \intertext{and}
          \|\nabla \ure^{\sigma \Delta} \varphi\|_{L^\theta(\Omega)}
    &\le  c_1 \|\nabla \varphi\|_{L^\infty(\Omega)} \quad \text{for all $\varphi \in L^\infty(\Omega)$ and all $\sigma \gt 0$}.
  \end{align*}
  
  Therefore, for $t \in (0, \tmax)$ we have by \eqref{eq:tilde_repr_v} and Lemma~\ref{lm:positivity_mass_const}
  \begin{align*}
    &\pe  \|\nabla v(\cdot, t)\|_{L^{\theta}(\Omega)} \\
    &\le  \|\nabla \ure^{t \Delta} v_0\|_{L^{\theta}(\Omega)} 
          + \int_0^t \| \nabla \ure^{(t-s) \Delta} v(\cdot, s) w(\cdot, s) \|_{L^{\theta}(\Omega)} \ds \\
    &\le  c_1 \|\nabla v_0\|_{L^\infty(\Omega)}
          + c_2 \int_0^t (1 + (t-s)^\gamma \ure^{-\lambda (t-s)} )\|v(\cdot, s) w(\cdot, s)\|_{L^p(\Omega)} \ds \\
    &\le  c_1 \|\nabla v_0\|_{L^\infty(\Omega)}
          + c_2 \|v_0\|_{L^{\infty}(\Omega))} \|w\|_{L^\infty((0, \tmax); L^{p}(\Omega))}
            \int_0^\infty (1 + s^\gamma) \ure^{-\lambda s} \ds,
  \end{align*}
  where $\gamma \gt -1$ warrants finiteness of the last integral therein.
\end{proof}

\begin{lemma} \label{lm:lp_lq_u}
  Let $p, \theta, q \in [1, \infty]$ with
  \begin{align*}
    \frac1p + \frac{1}{\theta} = \frac{1}{q}.
  \end{align*}
  
  For all $p' \gt 1$ with
  \begin{align*}
    p'
    \begin{cases}
      \lt \frac{nq}{n-q}, & q \le n, \\
      \le \infty,         & q \gt n
    \end{cases}
  \end{align*}
  there exists $C \gt 0$ such that
  \begin{align*}
        \|u\|_{L^\infty((0, \tmax); L^{p'}(\Omega))}
    \le C \left(1 + \|u\|_{L^\infty((0, \tmax); L^p(\Omega))} \|\nabla v\|_{L^\infty((0, \tmax); L^{\theta}(\Omega))} \right)
  \end{align*}
  holds.
\end{lemma}
\begin{proof}
  Without loss of generality we assume $p' \ge q$.
  Define $\gamma \gt -1$ as in \eqref{eq:lp_lq_nabla_v:gamma} with $q$ instead of $p$ and $p'$ instead of $\theta$.

  Again relying on known smoothing estimates for the Neumann Laplace semigroup
  (cf.\ \cite[Lemma~1.3~(i) and (iv)]{WinklerAggregationVsGlobal2010})
  we can find $c_1, c_2, \lambda \gt 0$ such that
  \begin{align*}
          \|\ure^{\sigma \Delta} \varphi\|_{L^{p'}(\Omega)}
    &\le  c_1 \|\varphi\|_{L^{p'}(\Omega)} \quad \text{for all $\varphi \in L^{p'}(\Omega)$ and all $\sigma \gt 0$}
  \intertext{and}
          \|\ure^{\sigma \Delta} \nabla \cdot \varphi\|_{L^{p'}(\Omega)}
    &\le  c_2 (1 + \sigma^\gamma) e^{-\lambda \sigma} \|\varphi\|_{L^q(\Omega)} \quad \text{for all $\varphi \in L^q(\Omega)$ and all $\sigma \gt 0$},
  \end{align*}
  hence by \eqref{eq:tilde_repr_u} we have for $t \in (0, \tmax)$
  \begin{align*}
          \|u(\cdot, t)\|_{L^{p'}(\Omega)}
    &\le  \|\ure^{t \Delta} u_0\|_{L^{p'}(\Omega)} + \int_0^t \|\ure^{(t-s) \Delta} \nabla \cdot (u(\cdot, t) \nabla v(\cdot, t))\|_{L^{p'}(\Omega)} \ds\\
    &\le  c_1\|u_0\|_{L^{p'}(\Omega)}
          + c_2 \int_0^t (1 + (t-s)^\gamma) \ure^{\lambda (t-s)} \|u(\cdot, t) \nabla v(\cdot, t)\|_{L^q(\Omega)} \\
    &\le  c_1\|u_0\|_{L^{p'}(\Omega)}
          + c_2 \|u\|_{L^\infty((0, \tmax); L^p(\Omega))} \|\nabla v\|_{L^\infty((0, \tmax); L^{\theta}(\Omega))}
            \int_0^\infty (1 + s^\gamma) \ure^{-\lambda s} \ds.
  \end{align*}
  Finiteness of the last integral therein is again guaranteed by $\gamma \gt - 1$.
\end{proof}
 
Equipped with these estimates we are able to improve our extensibility criterion of Lemma~\ref{lm:local_ex}.
\begin{prop} \label{prop:global_ex_n}
  Let $p \ge \max\{2, n\}$. If there exists $c_0 \gt 0$ such that
  \begin{align*}
    \intom u^p \lt c_0 \quad \text{in $[0, \tmax)$}
  \end{align*}
  then $\tmax = \infty$.
  In that case we furthermore have
  \begin{align} \label{eq:global_ex_n:bounds}
    \{(u(\cdot, t), v(\cdot, t), w(\cdot, t)) : t \in (0, \tmax)\}
    \text{ is bounded in }
    L^\infty(\Omega) \times W^{1, \infty}(\Omega) \times L^\infty(\Omega).
  \end{align}
\end{prop}
\begin{proof}
  In view of Lemma~\ref{lm:local_ex} it suffices to show \eqref{eq:global_ex_n:bounds}.

  Set $\theta \defs 2p$.
  As $\frac{np}{(n-p)_+} = \infty$,
  Lemma~\ref{lm:bound_u_bound_w} and Lemma~\ref{lm:lp_lq_nabla_v} assert boundedness
  of $\{v(\cdot, t): t \ge 0\}$ in $W^{1, \theta}(\Omega)$.

  Since
  \begin{align*}
        q
    \defs  \frac{p \theta}{p + \theta}
     =  \frac{2p^2}{3p}
     =  \frac{2}{3} p
    \gt \max\left\{1, \frac{n}{2}\right\}
  \end{align*}
  fulfills $\frac{1}{q} = \frac{1}{p} + \frac{1}{\theta}$ and
  \begin{align*}
          \frac{nq}{(n-q)_+}
    &\gt  \frac{n \cdot \frac{n}{2}}{n - \frac{n}{2}}
    =     n
  \end{align*}
  we may invoke Lemma~\ref{lm:lp_lq_u} to obtain boundedness of $S \defs \{u(\cdot, t) : t \ge 0\}$ in $L^{p'}(\Omega)$ for some $p' \gt n$.

  By again using Lemma~\ref{lm:bound_u_bound_w} and Lemma~\ref{lm:lp_lq_nabla_v} we obtain boundedness of
  $\{v(\cdot, t): t \ge 0\}$ in $W^{1, \infty}(\Omega)$.
  Then $q' \defs p'$ fulfills $\frac{1}{p'} + \frac{1}{\infty} = \frac{1}{q'}$ and $q' \gt n \ge 1$,
  hence Lemma~\ref{lm:lp_lq_u} implies boundedness of $S$ in $L^\infty(\Omega)$.
  Therefore the statement follows by a final application of Lemma~\ref{lm:bound_u_bound_w}.
\end{proof}

\subsection{Global existence for small $\|v_0\|_{L^\infty(\Omega)}$}
The estimates in this as well as in the following subsection will rely heavily on the following functional inequality.
\begin{lemma} \label{lm:u_p_varphi_v_estimate}
  Let $p \in (0, 1) \cup (1, \infty)$, $\sigma_p \defs \sign(p-1)$
  and $\varphi \in C^2([0, \|v_0\|_{L^\infty(\Omega)}])$ with $\varphi \gt 0$.
  Then for all $\eta_1, \eta_2 \gt 0$
  \begin{align*}
    &\pe \frac{\sigma_p}{p} \frac{\diff}{\dt} \intom u^p \varphi(v)
         +|p-1| \intom u^{p-2} |\nabla u|^2 \varphi(v) [1 - \eta_1 - \eta_2] \\
    &\le \intom u^p |\nabla v|^2 g_\varphi(v)
         - \frac{\sigma_p}{p} \intom u^p v \varphi'(v) w
  \end{align*}
  holds, where
  \begin{align*}
        g_\varphi(v)
    \defs  \frac{|p-1|}{4 \eta_2} \varphi(v)
        + |\varphi'(v)|
        + \frac{1}{\eta_1 |p-1|} \frac{\varphi'^2(v)}{\varphi(v)}
        - \frac{\sigma_p}{p} \varphi''(v).
  \end{align*}
\end{lemma}
\begin{proof}
  By integrating by parts we obtain
  \begin{align*}
         \frac{\sigma_p}{p} \frac{\diff}{\dt} \intom u^p \varphi(v)
    &=   -\sigma_p \intom \nabla(u^{p-1} \varphi(v)) \cdot \nabla u
         +\sigma_p \intom \nabla(u^{p-1} \varphi(v)) \cdot  (u \nabla v) \\
    &\pe -\frac{\sigma_p}{p} \intom \nabla(u^p \varphi'(v)) \cdot \nabla v
         - \frac{\sigma_p}{p} \intom u^p v \varphi'(v) w \\
    &=   -|p-1| \intom u^{p-2} |\nabla u|^2 \varphi(v) - \sigma_p \intom u^{p-1} \varphi'(v) \nabla u \cdot \nabla v \\
    &\pe + |p-1| \intom u^{p-1} \varphi(v) \nabla u \cdot \nabla v + \sigma_p \intom u^p |\nabla v|^2 \varphi'(v) \\
    &\pe - \sigma_p \intom u^{p-1} \varphi'(v) \nabla u \cdot \nabla v - \frac{\sigma_p}{p} \intom u^p |\nabla v|^2 \varphi''(v) \\
    &\pe - \frac{\sigma_p}{p} \intom u^p v \varphi'(v) w
  \end{align*}
  in $(0, \tmax)$.
  
  Herein we use Young's inequality
  to conclude
  \begin{align*}
          \left| -2 \sigma_p \intom u^{p-1} \varphi'(v) \nabla u \cdot \nabla v \right|
    &\le  \eta_1 |p-1| \intom u^{p-2} |\nabla u|^2 \varphi(v)
          + \frac{1}{\eta_1 |p-1|} \intom u^p |\nabla v|^2 \cdot \frac{\varphi'^2(v)}{\varphi(v)}
  \intertext{and}
          \left| |p-1| \intom u^{p-1} \varphi(v) \nabla u \cdot \nabla v \right|
    &\le  \eta_2 |p-1| \intom u^{p-2} |\nabla u|^2 \varphi(v)
          + \frac{|p-1|}{4 \eta_2} \intom u^p |\nabla v|^2 \varphi(v)
  \end{align*}
  in $(0, \tmax)$.
  
  Combing these estimates with $\sigma_p \varphi' \le |\varphi'|$ already completes the proof.
\end{proof}

A first application of this Lemma is
\begin{prop} \label{prop:global_ex_small_v0}
  If $v_0 \le \frac{1}{3\max\{2, n\}}$, then $\tmax = \infty$ and \eqref{eq:global_ex_n:bounds} holds.
\end{prop}
\begin{proof}
  We follow an idea of \cite[Lemma~3.1]{TaoBoundednessChemotaxisModel2011}.

  Without loss of generality suppose $v_0 \not\equiv 0$.
  Let $p \defs \max\{2, n\}$, $I \defs [0, \|v_0\|_{L^\infty(\Omega)}]$ and 
  \begin{align*}
    \varphi: I \ra \R, \quad s \mapsto \ur\ure^{\gamma s^2},
  \end{align*}
  where
  \begin{align*}
    \gamma \defs \frac{p-1}{12 p \|v_0\|_{L^\infty(\Omega)}^2} \gt 0.
  \end{align*}
  
  Then we have
  \begin{align*}
    \varphi'(s) = 2\gamma s \ur\ure^{\gamma s^2} \ge 0
    \quad \text{and} \quad
    \varphi''(s) = [2\gamma + (2\gamma s)^2 ] \ur\ure^{\gamma s^2} \ge 2\gamma \ur\ure^{\gamma s^2} \gt 0
  \end{align*}
  for $s \in I$,
  such that Lemma~\ref{lm:u_p_varphi_v_estimate} yields for $\eta_1 = \eta_2 = \frac12$
  \begin{align*}
         \frac1p \frac{\diff}{\dt} \intom u^p \varphi(v)
    &\le \frac1p \intom u^p |\nabla v|^2 
           \left[
             \frac{p(p-1)}{2} \varphi(v)
             + p \varphi'(v)
             + \frac{2p}{p-1} \frac{\varphi'^2(v)}{\varphi(v)}
             - \varphi''(v)
           \right].
  \end{align*}

  As $0 \le v \le \|v_0\|_{L^\infty(\Omega)}$ by Lemma~\ref{lm:positivity_mass_const}, we have
  \begin{align*}
          \frac{p(p-1) \varphi(v)}{2 \varphi''(v)}
    &\le  \frac{p(p-1)}{4 \gamma}
    =     3 p^2 \|v_0\|_{L^\infty(\Omega)}^2
    \le   \frac13
    \quad \text{in $(0, \tmax)$}, \\
          \frac{p \varphi'(v)}{\varphi''(v)}
    &\le  \frac{2p \gamma v \ure^{\gamma v^2}}{2\gamma \ure^{\gamma v^2}}
    \le   p\|v_0\|_{L^\infty(\Omega)}
    \le   \frac13
    \quad \text{in $(0, \tmax)$}
    \intertext{and}
          \frac{2p \varphi'^2(v)}{(p-1) \varphi(v) \varphi''(v)}
    &\le  \frac{2p (2\gamma v)^2}{2(p-1)\gamma}
    \le   \frac{4p \gamma}{p-1} \|v_0\|_{L^\infty(\Omega)}^2
    =     \frac13
    \quad \text{in $(0, \tmax)$},
  \end{align*}
  hence
  \begin{align*}
    \frac{\diff}{\dt} \intom u^p \varphi(v) \le 0 \quad \text{in $(0, \tmax)$}.
  \end{align*}
  
  Since $\varphi \ge 1$ we obtain upon integrating
  \begin{align*}
        \intom u^p
    \le \intom u^p \varphi(v)
    \le \intom u_0^p \varphi(v_0)
    \quad \text{in $(0, \tmax)$},
  \end{align*}
  therefore we may apply Proposition~\ref{prop:global_ex_n} to obtain the statement.
\end{proof}

\subsection{Global existence for $n \le 2$} \label{sec:global_ex_n2}
The following lemma exploits the special structure of \eqref{prob:p} and is a key ingredient for our further analysis.
Space-time bounds of $u$ can be turned into space bounds of $w$:
\begin{lemma} \label{lm:spacetime_u_space_w}
  For all $p \in [1, \infty)$ there exists $C \gt 0$ such that for all $T \in (0, \tmax]$ we have
  \begin{align*}
    \sup_{t \in (0, T)} \|w(\cdot, t)\|_{L^p(\Omega)} \le C \left( 1 + \|u\|_{L^p(\Omega \times (0, T))} \right).
  \end{align*}
\end{lemma}
\begin{proof}
  Fix $p \in [1, \infty)$ and let $p' \defs \frac{p}{p-1}$.
  By using \eqref{eq:tilde_repr_w} and Hölder's inequality we obtain 
  \begin{align*}
          \|w(\cdot, t)\|_{L^p(\Omega)}
    &\le  \ure^{-\delta t} \|w_0\|_{L^p(\Omega)}
          + \int_0^t e^{-\delta(t-s)} \|u(\cdot, s)\|_{L^p(\Omega)} \ds \\
    &\le  \|w_0\|_{L^p(\Omega)}
          + \|u\|_{L^p(\Omega \times (0, t))} \left( \int_0^t e^{-\delta p' s} \ds \right)^\frac{1}{p'} \\
    &\le  \|w_0\|_{L^p(\Omega)}
          + \|u\|_{L^p(\Omega \times (0, T))} (\delta p')^{-\frac{1}{p'}}
  \end{align*}
  for $t \in (0, T)$,
  hence the statement follows for $C \defs \max\{\|w_0\|_{L^p(\Omega)}, (\delta p')^{-\frac{1}{p'}}\}$.
\end{proof}
 
We proceed to gain space-time bounds for $u$:

\begin{lemma} \label{lm:nabla_u_p_space_time}
  There exists $p_0 \in (0, 1)$ such that for all $p \in (0, p_0)$ there is $C \gt 0$ with
  \begin{align*}
    \int_0^{\tmax} \intom |\nabla u^{\frac{p}{2}}|^2 \lt C.
  \end{align*}
\end{lemma}
\begin{proof}
  The function $\varphi: I \defs [0, \|v_0\|_{L^\infty(\Omega)}] \ra \R$ defined by
  \begin{align*}
    \varphi(s) \defs 1 + \|v_0\|_{L^\infty(\Omega)}^2 - s^2, \quad s \in I,
  \end{align*}
  satisfies
  \begin{align*}
    \varphi \ge 1,
    \quad
    \varphi' \le 0
    \quad \text{and} \quad
    \varphi'' \le -1.
  \end{align*}

  Therefore, for all $s \in I$ we have
  \begin{align*}
          g_p(s)
    &\defs   \frac{3}{4}(1-p) \varphi(s) + |\varphi'(s)| + \frac{3}{1-p} \frac{(\varphi')^2(s)}{\varphi(s)} + \frac1p \varphi''(s) \\
    &\le  \frac{3}{4}(1-p) \|\varphi\|_{L^\infty(I)} + \|\varphi'\|_{L^\infty(I)} + \frac{3}{1-p} \|\varphi'\|_{L^\infty(I)}^2 - \frac1p 
    \ra -\infty
  \end{align*}
  as $p \sea 0$,
  hence there exists $p_0 \in (0, 1)$ such that $g_p(s) \le 0$ for all $s \in I$ and all $p \lt p_0$.

  Lemma \ref{lm:u_p_varphi_v_estimate} with $\eta_1 = \eta_2 = \frac13$ then yields for $p \in (0, p_0)$
  \begin{align*}
          -\frac1p \frac{\diff}{\dt} \intom u^p \varphi(v)
          +\frac{1-p}{3} \intom u^{p-2} |\nabla u|^2 \varphi(v)
    &\le  \intom u^p |\nabla v|^2 g_p(v)
    \le   0
    \quad \text{in $(0, \tmax)$},
  \end{align*}
  as $v(\ombar \times [0, \tmax)) \subset I$  by Lemma~\ref{lm:positivity_mass_const}.
  
  Thus, upon integrating over $(0, T)$, $T \in (0, \tmax)$, we obtain by using Hölder's inequality
  \begin{align*}
          \frac{1-p}{3} \int_0^T \intom u^{p-2} |\nabla u|^2 
    &\le  \frac{1-p}{3} \int_0^T \intom u^{p-2} |\nabla u|^2 \varphi(v) \\
    &\le  \frac1p \intom u^p(\cdot, T) \varphi(v) - \frac1p \intom u_0^p \varphi(v) \\
    &\le  \frac1p \|\varphi\|_{L^\infty(I)} \intom u^p(\cdot, T)\\
    &\le  \frac1p \|\varphi\|_{L^\infty(I)} m^p |\Omega|^{1-p},
  \end{align*}
  hence the statement follows by setting $C \defs \frac{3}{(1-p)p} \|\varphi\|_{L^\infty(I)} m^p |\Omega|^{1-p}$
  and using the monotone convergence theorem.
\end{proof}

\begin{lemma} \label{lm:u_r_u_nabla_q}
  For $p \in (0, 2]$ there exists $C \gt 0$ with
  \begin{align*}
    \intom u^{\frac{2}{n}+p} \le C \left(1 + \intom |\nabla u^{\frac{p}{2}}|^2 \right) \quad \text{in $(0, \tmax)$}.
  \end{align*}
\end{lemma}
\begin{proof}
  Set $p' \defs  \frac{2}{p} \cdot (\frac{2}{n} + p) \gt 2$.
  As
  \begin{align*}
      \frac{2}{p'} \in (0, 1)
    \quad \text{and} \quad
      \left( \frac{1}{2} - \frac{1}{n} \right) \cdot \frac{2}{p'} + \frac{1 - \frac{2}{p'}}{\frac{2}{p}}
    = \frac1{p'} \left( \frac{n-2}{n} + \frac{p}{2} \cdot p' - p \right)
    = \frac1{p'}
  \end{align*}
  we may invoke the Gagliardo--Nirenberg inequality to obtain $C_1, C_2 \gt 0$ such that
  \begin{align*}
          \intom \psi^{\frac{2}{n} + p} 
    &=    \|\psi^\frac{p}{2}\|_{L^{p'}}^{p'} \\
    &\le  C_1 \|\nabla \psi^\frac{p}{2}\|_{L^2(\Omega)}^2 \|\psi^\frac{p}{2}\|_{L^{\frac{2}{p}}(\Omega)}^{p'-2}
          + C_2 \|\psi^\frac{p}{2}\|_{L^\frac{2}{p}(\Omega)}^{p'} \\
    &=    C_1 \int_\Omega |\nabla \psi^\frac{p}{2}|^2 \left( \intom \psi \right)^{\frac{p (p'-2)}{2}}
          + C_2 \left( \intom \psi \right)^{\frac{p p'}{2}}
  \end{align*}
  holds for all nonnegative $\psi: \Omega \ra \R$ with $\psi^\frac{p}{2} \in W^{1, 2}(\Omega)$.
  
  The statement follows then by taking $\psi = u(\cdot, t)$, $t \in (0, \tmax)$, and Lemma~\ref{lm:positivity_mass_const}.
\end{proof}

\begin{lemma} \label{lm:u_p_estimate}
  For $p \in (0, 2]$ there exists $C_p \gt 0$ such that
  \begin{align*}
          \frac{\sigma_p}{p} \frac{\diff}{\dt} \intom u^p
          +\frac{|p-1|}{p^2} \intom |\nabla u^\frac{p}{2}|^2
    &\le  C_p \intom |\nabla v|^{2+np} + C_p
    \quad \text{in $(0, \tmax)$},
  \end{align*}
  where $\sigma_p \defs \sign(p-1)$.
\end{lemma}
\begin{proof}
  Without loss of generality let $p \in (0, 2] \setminus \{1\}$.
  By applying Lemma~\ref{lm:u_p_varphi_v_estimate} with $\varphi \equiv 1$ and $\eta_1 = \eta_2 = \frac14$ we obtain
  \begin{align} \label{eq:u_p_estimate:func_ineq}
          \frac{\sigma_p}{p} \frac{\diff}{\dt} \intom u^p
          +\frac{|p-1|}{2} \intom u^{p-2} |\nabla u|^2
    &\le  |p-1| \intom u^p |\nabla v|^2
    \quad \text{in $(0, \tmax)$}.
  \end{align}

  According to Lemma~\ref{lm:u_r_u_nabla_q} we may find $C_p' \gt 0$ such that
  \begin{align*}
    \intom u^{\frac{2}{n} + p} \le C_p' \intom |\nabla u^{\frac{p}{2}}|^2 + C_p'
    \quad \text{in $(0, \tmax)$},
  \end{align*}
  hence Young's inequality (with exponents $\frac{\frac{2}{n}+p}{p}, \frac{\frac{2}{n}+p}{\frac{2}{n}} = \frac{2+np}{2}$)
  implies the existence of $C_p'' \gt 0$ satisfying
  \begin{align} \label{eq:u_p_estimate:u_p_nabla_v}
          \intom u^p |\nabla v|^2
    &\le  \frac{1}{p^2 C_p'} \intom u^{\frac{2}{n}+p} + C_p'' \intom |\nabla v|^{2 + np} \notag\\
    &\le  \frac{1}{p^2} \intom |\nabla u^{\frac{p}{2}}|^2 + \frac{1}{p^2} + C_p'' \intom |\nabla v|^{2 + np}
  \end{align}
  in $(0, \tmax)$.
  The statement follows by combining \eqref{eq:u_p_estimate:func_ineq} with \eqref{eq:u_p_estimate:u_p_nabla_v},
  due to the pointwise equality $|\nabla u^\frac{p}{2}|^2 = \frac{p^2}{4} u^{p-2} |\nabla u|^2$
  and by setting $C_p \defs |p-1| \max\{\frac{1}{p^2}, C_p''\}$.
\end{proof}

For $n \le 2$ we may now apply a bootstrap procedure to achieve globality of $(u, v, w)$.
A combination of Lemma~\ref{lm:nabla_u_p_space_time}, Lemma~\ref{lm:u_r_u_nabla_q} and Lemma~\ref{lm:spacetime_u_space_w}
serves as a starting point,
while Lemma~\ref{lm:u_p_estimate} and Lemma~\ref{lm:spacetime_u_space_w} are the main ingredients
for improving bounds for $u$ step by step.

\begin{prop} \label{prop:global_ex_n2}
  If $n \le 2$, then $\tmax = \infty$.
\end{prop}
\begin{proof}
  Suppose $\tmax \lt \infty$.
  
  By Lemma~\ref{lm:nabla_u_p_space_time} we may find $p_0 \in (0, 1)$ such that
  \begin{align*}
    \int_0^{\tmax} \intom |\nabla u^\frac{p_0}{2}|^2 \lt C_1
  \end{align*}
  for some $C_1 \gt 0$.

  As $\frac{2}{n} \ge 1$ a combination of Lemma~\ref{lm:u_r_u_nabla_q} and the Hölder inequality implies
  \begin{align*}
    \intom u^{1+p_0} \le C_2 \left(1 + \intom |\nabla u^\frac{p_0}{2}|^2 \right)
    \quad \text{in $(0, \tmax)$}
  \end{align*}
  for some $C_2 \gt 0$,
  we conclude
  \begin{align*}
    \int_0^{\tmax} \intom u^{1+p_0} \le C_2 \left( \tmax + C_1 \right).
  \end{align*}
  Hence, by Lemma~\ref{lm:spacetime_u_space_w} there is $c_0 \gt 0$ with
  \begin{align} \label{eq:global_ex_n2:w_ck0}
    \intom w^{1 + p_0} \le c_0 \quad \text{in $(0, \tmax)$}.
  \end{align}

  Set $a_0 \defs 2 - (1 + p_0) \in (0, 1)$ and
  \begin{align*}
    p_{k+1} \defs
    \begin{cases}
      \frac{p_k}{a_0},    & \frac{p_k}{a_0} \lt 1, \\
      \frac{p_k + 1}{2},  & \frac{p_k}{a_0} \ge 1
    \end{cases}
  \end{align*}
  for $k \in \N_0$.
  Note that $p_k \in (0, 1)$ and $p_k \ra 1$ for $k \ra \infty$,
  hence there exists $k_0 \in \N$ with $p_{k_0} \gt \frac12$.

  We next show by induction that for each $k \in \N_0$ there exists $c_k \gt 0$ such that
  \begin{align} \label{eq:global_ex_n2:w_ck}
    \intom w^{1+p_k} \le c_k \quad \text{in $(0, \tmax)$}.
  \end{align}

  Let $k = 0$, then \eqref{eq:global_ex_n2:w_ck} is exactly \eqref{eq:global_ex_n2:w_ck0},
  hence suppose \eqref{eq:global_ex_n2:w_ck} holds for some $k \in \N_0$.
  As
  \begin{align*}
          2 + n p_{k+1}
    &\le  2 (1 + p_{k+1}) \\
    &\le  2 \left(1 + \frac{p_k}{a_0} \right) \\
    &\lt  2 \left(1 + \frac{1 - a_0 + p_k}{a_0} \right) \\
    &=    2 + \frac{2(1 + p_k) - 2a_0}{a_0} \\
    &=    \frac{2(1 + p_k)}{2 - (1 + p_0)} \\
    &\lt  \frac{2(1 + p_k)}{2 - (1 + p_k)}
     \le  \frac{n(1 + p_k)}{(n - (1 + p_k))_+}
  \end{align*}
  we may apply Lemma~\ref{lm:lp_lq_nabla_v} to obtain $c_{k+1}' \gt 0$ with
  \begin{align*} 
    \intom |\nabla v|^{2 + n p_{k+1}} \le c_{k+1}' \quad \text{in $(0, \tmax)$}.
  \end{align*}

  By applying Lemma~\ref{lm:u_p_estimate} and integrating over $(0, T)$ for $T \in (0, \tmax)$ we then obtain
  \begin{align} \label{eq:global_ex_n2:cor_int}
        -\frac{1}{p_{k+1}} \intom u^{p_{k+1}}(\cdot, T) + \frac{1}{p_{k+1}} \intom u_0^{p_{k+1}}
        + \frac{1 - p_{k+1}}{p_{k+1}^2} \int_0^T \intom |\nabla u^\frac{p_{k+1}}{2}|^2
    \le C_{p_{k+1}} (1 + c_k') T
  \end{align}
  with $C_{p_{k+1}}$ as in Lemma~\ref{lm:u_p_estimate}.
  
  Since $p_{k+1} \in (0, 1)$ we have $\intom u^{p_{k+1}} \lt c_{k+1}''$ in $(0, \tmax)$
  for some $c_{k+1}'' \gt 0$ by Lemma~\ref{lm:positivity_mass_const}.
  As $u_0 \ge 0$ and $\tmax \lt \infty$ by assumption, \eqref{eq:global_ex_n2:cor_int} implies
  \begin{align*} 
        \int_0^T \intom |\nabla u^\frac{p_{k+1}}{2}|^2
    \le \frac{p_{k+1}^2}{1 - p_{k+1}} \left((1 + c_{k+1}') C_{p_{k+1}} \tmax + \frac{c_{k+1}''}{p_{k+1}}\right)
    \lt \infty
    \quad \text{for all $T \in (0, \tmax)$}.
  \end{align*}
  By using Lemma~\ref{lm:u_r_u_nabla_q} and Lemma~\ref{lm:spacetime_u_space_w}
  we then obtain \eqref{eq:global_ex_n2:w_ck} for $k+1$ instead of $k$ and some $c_{k+1} \gt 0$.
   
  Finally, \eqref{eq:global_ex_n2:w_ck} for $k = k_0$
  and Lemma~\ref{lm:lp_lq_nabla_v} assert boundedness of $\{\nabla v(\cdot, t): t \in (0, \tmax\}$ in $L^6(\Omega)$,
  since
  \begin{align*}
        6
    =   \frac{2(1+\frac12)}{2 - (1 + \frac12)}
    \lt \frac{2(1+p_{k_0})}{2 - (1 + p_{k_0})}
    \le \frac{n(1+p_{k_0})}{(n - (1 + p_{k_0}))_+},
  \end{align*}
  such that by another application of Lemma~\ref{lm:u_p_estimate} and Hölder's inequality ($2 + 2n \le 6$) we obtain $C \gt 0$ with
  \begin{align*}
    \ddt \intom u^2 \lt C \quad \text{in $(0, \tmax)$}.
  \end{align*}
  However, this contradicts Proposition~\ref{prop:global_ex_n}.
\end{proof}

\begin{remark}
  Apart from Proposition~\ref{prop:global_ex_n2} all statements in this subsection hold for $n \in \N$.
  However, for $n \ge 3$ the lemmata above (at least in the form stated)
  are not sufficient to prove $\tmax = \infty$ also for higher dimensions:
  Lemma~\ref{lm:nabla_u_p_space_time} and Lemma~\ref{lm:u_r_u_nabla_q} imply
  \begin{align*}
    u^{\frac{2}{n} + p} \lt C(T) \quad \text{in $(0, T)$}
  \end{align*}
  for $T \in (0, \tmax)$ and some $p \in (0, 1), C(T) \gt 0$,
  but for $p \lt \frac{n-2}{n}$ this does not improve on boundedness in $L^1(\Omega)$,
  which is already known (Lemma~\ref{lm:positivity_mass_const}).
\end{remark}

Theorem~\ref{prob:keller_segel} is now an immediate consequence of the propositions above:
\begin{proof}[Proof of Theorem~\ref{prob:keller_segel}]
  Local existence and uniqueness have been shown in Lemma~\ref{lm:local_ex},
  while $\tmax = \infty$ has been proved in Proposition~\ref{prop:global_ex_small_v0} and Proposition~\ref{prop:global_ex_n2}
  for the cases $\|v_0\|_{L^\infty(\Omega)} \le \frac{1}{3 \max\{2, n\}}$ and $n \le 2$, respectively.
\end{proof}

\section{Large time behavior} \label{sec:large_time}
\subsection{A sufficient condition}
We show convergence of the solution towards a spatial constant equilibrium, if additionally the following condition is satisfied.
That is, if one is able to show this for a set of parameters not discussed here, the statements in the following subsections still apply.
\begin{cond} \label{cond:large_time}
  The solution $(u, v, w)$ is global in time and there exist $\theta \gt n$ and $C \gt 0$ with
  \begin{align*}
    \intom |\nabla v|^{\theta} \lt C \quad \text{in $(0, \infty)$}.
  \end{align*}
\end{cond}
In the remainder of this subsection we will show that if $n \le 2$ or $n \ge 2$ and $\|v_0\|_{L^\infty(\Omega)} \le \frac1{3n}$
this is always the case.

\begin{lemma} \label{lm:gn_u_p_u*}
  Let $n \le 2, p \in (0, 1)$ and set $u^* \defs \frac{1}{|\Omega|} \intom u^{\frac{p}{2}}$.
  Then there exists $C \gt 0$ such that
  \begin{align*}
    \intom \left| u^{\frac{p}{2}} - u^* \right|^{\frac{2}{p} \cdot (1 + p)} \le C \intom |\nabla u^\frac{p}{2}|^2
    \quad \text{in $(0, \infty)$.}
  \end{align*}
\end{lemma}
\begin{proof}
  Set $p' \defs \frac{2}{p} \cdot (\frac{2}{n} + p) \gt 2$.
  As
  \begin{align*}
      \frac{2}{p'} \in (0, 1)
    \quad \text{and} \quad
      \left( \frac{1}{2} - \frac{1}{n} \right) \cdot \frac{2}{p'} + \frac{1 - \frac{2}{p'}}{\frac{2}{p}}
    = \frac1{p'} \left( \frac{n-2}{n} + \frac{p}{2} \cdot p' - p \right)
    = \frac1{p'}
  \end{align*}
  
  By the Gagliardo-Nirenberg inequality there exist $C_1, C_2 \gt 0$ such that
  \begin{align*}
          \|\psi\|_{L^{p'}}^{p'}
     \le  C_1 \|\nabla \psi\|_{L^2(\Omega)}^2 \|\psi\|_{L^{\frac{2}{p}}(\Omega)}^{p'-2}
          + C_2 \|\psi\|_{L^2(\Omega)}^{p'}
    \quad \text{for all $\psi \in W^{1, 2}(\Omega)$}.
  \end{align*}

  Additionally, Poincar\'e's and Hölder's ($\frac{2}{p} \gt 2$, as $p \in (0, 1)$) inequalities yield the existence of $C_3 \gt 0$ such that
  \begin{align*}
        \|\psi\|_{L^2(\Omega)}^{p'} 
    =   \|\psi\|_{L^2(\Omega)}^{2} \|\psi\|_{L^2(\Omega)}^{p'-2} 
    \le C_3 \|\nabla \psi\|_{L^2(\Omega)}^{2} \|\psi\|_{L^\frac{2}{p}(\Omega)}^{p'-2}
    \quad \text{for all $\psi \in W^{1, 2}(\Omega)$ with $\intom \psi = 0$.}
  \end{align*}

  By combining these estimates we may find $C' \gt 0$ such that
  \begin{align} \label{eq:gn_u_p_u*:psi_comb}
          \|\psi\|_{L^{p'}}^{p'}
    \le   C' \|\nabla \psi\|_{L^2(\Omega)}^2 \|\psi\|_{L^{\frac{2}{p}}(\Omega)}^{p'-2}
    \quad \text{for all $\psi \in W^{1, 2}(\Omega)$ with $\intom \psi = 0$.}
  \end{align}

  Since $u^*$ is constant in space we have $\nabla u^\frac{p}{2} = \nabla (u^\frac{p}{2} - u^*)$.
  As boundedness of $u^\frac{p}{2} - u^*$ in $L^\frac{2}{p}(\Omega)$ is implied by Lemma~\ref{lm:positivity_mass_const}
  and the assumption $n \le 2$ warrants $p' \le \frac{2}{p} \cdot (1+p)$,
  by taking $\psi = u^{\frac{p}{2}}(\cdot, t) - u^*$, $t \in (0, \tmax)$, in \eqref{eq:gn_u_p_u*:psi_comb} and employing Hölder's inequality
  we obtain the statement.
\end{proof}

\begin{lemma} \label{lm:nabla_v_theta}
  Let $n \le 2$. Then there exist $\theta \gt n$ and $C \gt 0$ such that Condition~\ref{cond:large_time} is fulfilled.
\end{lemma}
\begin{proof}
  By Proposition~\ref{prop:global_ex_n2} the solution is global in time.

  Lemma~\ref{lm:nabla_u_p_space_time} allows us to choose $p \in (0, 1)$ and $C_0 \gt 0$ such that
  \begin{align*}
    \frac{2}{p} \in \N
    \quad \text{and} \quad
    \int_0^\infty \intom |\nabla u^{\frac{p}{2}}|^2 \lt C_0.
  \end{align*}
  Let $u^*$ be as in Lemma~\ref{lm:gn_u_p_u*} and set
  \begin{align*}
    \tilde w: \ombar \times [0, \infty) \ra \R, \quad
    (x, t) \mapsto w(x, t) - \int_0^t \ure^{-\delta (t-s)} \left( u(x, s) - \left| u^\frac{p}{2}(x, s) - u^*(s) \right|^\frac{2}{p} \right) \ds.
  \end{align*}
  The representation formula \eqref{eq:tilde_repr_w}, Hölder's inequality
  and Lemma~\ref{lm:gn_u_p_u*} yield for $t \in (0, \infty)$
  \begin{align} \label{eq:nabla_v_theta:c1}
          \intom |\tilde w(\cdot, t)|^{p+1}
    &=    \intom \left| \ure^{-\delta t} w_0 + \int_0^t \ure^{-\delta(t-s)} \left| u^\frac{p}{2}(\cdot, s) - u^*(s) \right|^\frac{2}{p} \ds \right|^{p+1} \notag \\
    &\le  2^{p} \ure^{-(p+1)\delta t} \intom w_0^{p+1}
          + 2^{p}  \intom \left( \int_0^t \ure^{-\delta (t-s)} \left| u^\frac{p}{2}(\cdot, s) - u^*(s) \right|^\frac{2}{p} \ds \right)^{p+1} \notag \\
    &\le  2^{p} \intom w_0^{p+1}+ 2^{p}
          \intom \left[ \left(\int_0^t \ure^{-\frac{(p+1)\delta}{p} (t-s)}\right)^{p} \int_0^t \left| u^\frac{p}{2} - u^* \right|^{\frac{2}{p} \cdot (1 + p)} \right] \notag \\
    &\le  2^{p} \intom w_0^{p+1} + 2^{p}
          \left(\int_{-\infty}^t \ure^{-\frac{(p+1)\delta}{p} (t-s)}\right)^{p} \int_0^\infty \intom \left| u^\frac{p}{2} - u^* \right|^{\frac{2}{p} \cdot (1 + p)} \notag \\
    &\le  2^{p} \intom w_0^{p+1} + 2^{p} \left(\frac{p}{(p+1)\delta}\right)^{p} \int_0^\infty C' \intom |\nabla u^\frac{p}{2}|^2 \notag \\
    &\le  2^{p} \intom w_0^{p+1} + 2^{p} \left(\frac{p}{(p+1)\delta}\right)^{p} C' C_0
    \sfed C_1.
  \end{align}
  
  As by Hölder's inequality and Lemma~\ref{lm:positivity_mass_const} there is $C_2 \gt 0$ with $|u^*| \le C_2$ in $(0, \tmax)$,
  we may further estimate (using the binomial theorem, note that $\frac{2}{p} \in \N$, and Jensen's inequality)
  \begin{align*}
          \left| u - \left| u^\frac{p}{2} - u^* \right|^\frac{2}{p} \right|^{1 + \frac{p}{2}}
    &\le  \left( \sum_{k=1}^{\frac{2}{p}} \binom{\frac{2}{p}}{k} u^{\frac{p}{2} \cdot (\frac{2}{p} - k)} (u^*)^k \right)^{1 + \frac{p}{2}} \hspace{-0.5cm}
    &\le  \left( \sum_{k=1}^{\frac{2}{p}} C_2^k \binom{\frac{2}{p}}{k} u^{1 - \frac{kp}{2}} \right)^{1 + \frac{p}{2}} \hspace{-0.2cm}
    &\le  \sum_{k=1}^{\frac{2}{p}} D_k u^{q_k}
  \end{align*}
  for certain $D_k \gt 0$ and
  \begin{align*}
        q_k
    \defs  \left(1 - \frac{kp}{2}\right) \left(1 + \frac{p}{2}\right)
    =   1 + (1-k) \frac{p}{2} - \frac{k p^2}{4}
    \lt 1
  \end{align*}
  for $k \in \{1, \dots, \frac{2}{p}\}$ in $(0, \tmax)$.

  Because of $\int_{-\infty}^t \delta \ure^{-\delta(t-s)} \ds = 1$ we may apply Jensen's inequality to further obtain that
  \begin{align} \label{eq:nabla_v_theta:c3}
          \intom |w - \tilde w|^{1 + \frac{p}{2}}
    &=    \intom \left| \int_0^t \ure^{-\delta (t-s)}
            \left( u(\cdot, s) - \left| u^\frac{p}{2}(\cdot, s) - u^*(s) \right|^\frac{2}{p} \right) \ds 
          \right|^{1 + \frac{p}{2}} \notag \\
    &\le  \delta^{-(1+\frac{p}{2})} \intom \left( \int_{-\infty}^t \delta \ure^{-\delta (t-s)} \mathds 1_{(0, t)}(s)
            \left| u(\cdot, s) - \left| u^\frac{p}{2}(\cdot, s) - u^*(s) \right|^\frac{2}{p} \right| \ds
          \right)^{1 + \frac{p}{2}} \notag \\
    &\le  \delta^{-(1+\frac{p}{2})} \int_{-\infty}^t \left( \delta \ure^{-\delta (t-s)} \mathds 1_{(0, t)}(s)
          \sum_{k=1}^{\frac{2}{p}} D_k \intom u^{q_k}(\cdot, s) \right) \ds
     \le  C_3
  \end{align}
  holds in $(0, \tmax)$ for some $C_3 \gt 0$ due to Lemma~\ref{lm:positivity_mass_const},
  as $q_k \lt 1$ for $k \in \{1, \dots, \frac{2}{p}\}$.
  
  As another application of Hölder's inequality gives
  \begin{align*}
        \intom |\tilde w|^{1 + \frac{p}{2}}
    \le C_4 \intom |\tilde w|^{1 + p}
  \end{align*}
  for some $C_4 \gt 0$, 
  we obtain by combining \eqref{eq:nabla_v_theta:c1} and \eqref{eq:nabla_v_theta:c3}
  \begin{align*}
        \|w(\cdot, t)\|_{L^{1 + \frac{p}{2}}(\Omega)}
    &\le \|\tilde w(\cdot, t)\|_{L^{1 + \frac{p}{2}}(\Omega)} + \|w(\cdot, t) - \tilde w(\cdot, t)\|_{L^{1 + \frac{p}{2}}(\Omega)}
     \le  (C_4 C_1)^\frac{1}{1+\frac{p}{2}} + C_3^\frac{1}{1+\frac{p}{2}}
  \end{align*}
  for $t \in (0, \tmax)$.
   
  The statement follows by applying Lemma~\ref{lm:lp_lq_nabla_v},
  as $\frac{2 (1 + \frac{p}{2})}{2 - (1 + \frac{p}{2})} \gt 2$
  (since $1+\frac{p}{2} \gt 1 - \frac{p}{2}$)
  and $\frac{1 (1 + \frac{p}{2})}{(1 - (1 + \frac{p}{2}))_+} \gt 1$.
\end{proof}

\begin{prop} \label{prop:cond_fulfilled}
  If $n \le 2$ or $n \ge 2$ and $\|v_0\|_{L^\infty(\Omega)} \le \frac{1}{3n}$, then Condition~\ref{cond:large_time} is fulfilled.
\end{prop}
\begin{proof}
  For $n \le 2$ this is a consequence of  Lemma~\ref{lm:nabla_v_theta} while 
  for $n \gt 2$ and $\|v_0\|_{L^\infty(\Omega)} \le \frac{1}{3n}$ this already has been shown in Proposition~\ref{prop:global_ex_small_v0}.
\end{proof}

\subsection{Convergence of $v$}
We begin by stating that $v$ converges at least in some very weak sense.
\begin{lemma} \label{lm:very_weak_conv_v}
  If $\tmax = \infty$, then
  \begin{align} \label{eq:very_weak_conv_v:vw}
      \int_0^\infty \intom v w \lt \infty.
  \end{align}
\end{lemma}
\begin{proof}
  Integrating the second equation in \eqref{prob:p} over $(0, T) \times \Omega$ (for any $T \gt 0$) yields
  \begin{align*}
      \intom v(\cdot, T) - \intom v_0
    = \int_0^T \intom \Delta v - \int_0^T \intom v w
  \end{align*}
  and as $\intom \Delta v = 0$ due to $\partial_\nu v = 0$ on $\partial \Omega$
  and $v \ge 0$ by Lemma~\ref{lm:positivity_mass_const} we have
  \begin{align*}
    \int_0^T \intom v w \le \intom v_0 \quad \text{for all } T \gt 0,
  \end{align*}
  which already implies \eqref{eq:very_weak_conv_v:vw}.
\end{proof}

\begin{lemma} \label{lm:convergence_v}
  If Condition~\ref{cond:large_time} is fulfilled, then $v(\cdot, t) \ra 0$ in $C^0(\ombar)$ for $t \ra \infty$.
\end{lemma}
\begin{proof}
  By \eqref{eq:very_weak_conv_v:vw} and since $vw \ge 0$
  there exists an increasing sequence $(t_k)_{k \in \N} \subset (0, \infty)$ such that 
  $t_k \ra \infty$ for $k \ra \infty$ with
  \begin{align} \label{eq:convergence_v:vw_tk_0}
    \intom v(\cdot, t_k) w(\cdot, t_k) \ra 0 \quad \text{for $k \ra \infty$.}
  \end{align}
  
  Condition~\ref{cond:large_time} and the embedding $W^{1, \theta}(\Omega) \embed \embed C^0(\ombar)$ for all $\theta \gt n$
  warrant that we may choose a subsequence of $(t_k)_{k \in \N}$ -- which we also denote by $(t_k)_{k \in \N}$ for convenience --
  along which
  \begin{align*}
    v(\cdot, t_k) \ra v_\infty \quad \text{in $C^0(\ombar)$ for $k \ra \infty$}
  \end{align*}
  for some $v_\infty \in C^0(\ombar)$.
  As $v \ge 0$ by Lemma~\ref{lm:positivity_mass_const} we have $v_\infty \ge 0$.

  \emph{Claim 1:}
    The limit $v_\infty$ is constant.

  Proof:
    Suppose $v_\infty$ is not constant, then
    \begin{align*}
          \lim_{k \ra \infty} \frac{1}{|\Omega|} \intom v(\cdot, t_k)
      =   \frac{1}{|\Omega|} \intom v_\infty
      \lt \|v_\infty\|_{L^\infty(\Omega)},
    \end{align*}
    hence there exists $k_0 \in \N$ such that
    \begin{align*}
          \frac{1}{|\Omega|} \intom v(\cdot, t_{k_0})
      \lt \|v_\infty\|_{L^\infty(\Omega)}.
    \end{align*}
    
    Set $\ol v(\cdot, t) \defs \ure^{t \Delta} v(\cdot, t_{k_0})$.
    It is well known (see for instance \cite[Lemma~1.3~(i)]{WinklerAggregationVsGlobal2010}) that
    \begin{align*}
            \left\| \ol v(\cdot, t) - \frac{1}{|\Omega|} \intom v(\cdot, t_{k_0})\right\|_{L^\infty(\Omega)}
      &\ra  0 \quad \text{as $t \ra \infty$},
    \end{align*}
    hence there exist $k_1 \gt k_0$ and $\eps \gt 0$ such that
    \begin{align*}
      \ol v(\cdot, t) \le \|v_\infty\|_{L^\infty(\Omega)} - \eps
    \end{align*}
    in $(t_{k_1} - t_{k_0}, \infty)$. Note that $t_{k_1} \ge t_{k_0}$ as $(t_k)_{k \in \N}$ is increasing.

    Moreover, $\ul v(\cdot, t) \defs v(\cdot, t + t_{k_0})$, $t \ge 0$, defines a subsolution $\ul v$ of
    \begin{align*}
      \begin{cases}
        v_t = \Delta v,                   & \text{in $\Omega \times (0, \infty)$}, \\
        \partial_\nu v = 0,               & \text{on $\partial \Omega \times (0, \infty)$}, \\
        v(\cdot, 0) = v(\cdot, t_{k_0}),  & \text{in $\Omega$},
      \end{cases}
    \end{align*}
    since $vw \ge 0$ by Lemma~\ref{lm:positivity_mass_const}.
    
    Therefore by comparison we have
    \begin{align*}
          v(\cdot, t+t_{k_0})
      =   \ul v(\cdot, t)
      \le \ol v(\cdot, t)
    \end{align*}
    for $t \ge 0$.
    However, this implies
    \begin{align*}
          \|v_\infty\|_{L^\infty(\Omega)} 
      =   \lim_{k_1 \le k \ra \infty} \|v(\cdot, t_k)\|_{L^\infty(\Omega)}
      \le \|v_\infty\|_{L^\infty(\Omega)} - \eps,
    \end{align*}
    which is a contradiction,
    hence $v_\infty$ is constant.

  \emph{Claim 2:}
    The limit fulfills $v_\infty \equiv 0$.

  Proof:
    Suppose $v_\infty \not\equiv 0$, then by the first claim $v_\infty \equiv C$ for some $C \gt 0$,
    thus we may choose $k_2 \in \N$ such that $v(\cdot, t_k) \ge \frac{C}{2}$ for all $k \ge k_2$.
    Then \eqref{eq:convergence_v:vw_tk_0} implies (as $w \ge 0$ by Lemma~\ref{lm:positivity_mass_const})
    \begin{align*}
          0
      \le \lim_{k_2 \le k \ra \infty} \intom w(\cdot, t_k)
      \le \lim_{k_2 \le k \ra \infty} \frac{2}{C} \intom v(\cdot, t_k) w(\cdot, t_k)
      =   0,
    \end{align*}
    hence
    \begin{align} \label{eq:convergence_v:lim_int_w}
      \lim_{k \ra \infty} \intom w(\cdot, t_k) = 0.
    \end{align}
    However, by \eqref{eq:tilde_repr_w} and Lemma~\ref{lm:positivity_mass_const} we have for $k \in \N$
    \begin{align*}
            \intom w(\cdot, t_k)
      &=    \ure^{-\delta t_k} \intom w_0 + \intom \int_0^{t_k} \ure^{-\delta (t_k-s)} u(\cdot, s) \ds
      \ge  \frac{m}{\delta} [1 - \ure^{-\delta t_k}]
      \ge  \frac{m}{\delta} [1 - \ure^{-\delta t_1}]
      \gt   0,
    \end{align*}
    which contradicts \eqref{eq:convergence_v:lim_int_w}; hence $v_\infty \equiv 0$.

  \emph{Claim 3:}
    The statement holds.

  Proof:
    Let $\eps \gt 0$.
    By the second claim we may choose $k' \in \N$ such that $\|v(\cdot, t_{k'})\|_{C^0(\ombar)} \lt \eps$.
    Therefore, Lemma~\ref{lm:positivity_mass_const} (for initial data $u(\cdot, t_{k'}), v(\cdot, t_{k'}), w(\cdot, t_{k'})$) implies
    \begin{align*}
      \|v(\cdot, t)\|_{C^0(\ombar)} \lt \eps \quad \text{for $t \ge t_{k'}$},
    \end{align*}
    thus the claim and hence the statement follow.
\end{proof}

\subsection{Boundedness of $u$}
Lemma~\ref{lm:convergence_v} allows us to show boundedness of $u$, which is an important step towards proving convergence.
\begin{lemma} \label{lm:u_bound_infty}
  If Condition~\ref{cond:large_time} is fulfilled, then $\{u(\cdot, t): t \ge 0\}$ is bounded in $L^\infty(\Omega)$.
\end{lemma}
\begin{proof}
  By Lemma~\ref{lm:convergence_v} there exists $t_0 \gt 0$ such that $v(\cdot, t_0) \le \frac1{3\max\{2, n\}}$.
  Proposition~\ref{prop:global_ex_small_v0} then states that the solution $(\tilde u, \tilde v, \tilde w)$ of \eqref{prob:p} with initial data
  \begin{align*}
    \tilde u_0 \defs u(\cdot, t_0), \quad \tilde v_0 \defs v(\cdot, t_0) \quad \text{and} \quad \tilde w_0 \defs w(\cdot, t_0)
  \end{align*}
  is bounded: There is $C \gt 0$ with $\tilde u(\cdot, t) \le C$ for all $t \gt 0$.
  As by uniqueness $\tilde u(\cdot, t) = u(\cdot, t + t_0)$
  and since $u \in C^0(\ombar \times [0, t_0])$ the statement follows.
\end{proof}

\newcommand{\bdd}[4]{%
\begin{align*}
  #1 \in C^{#2_{#3}, \frac{#2_{#3}}{2}}(\ombar \times [t', t'+1])
  \quad \text{with} \quad
  \|#1\|_{C^{#2_{#3}, \frac{#2_{#3}}{2}}(\ombar \times [t', t'+1])} \le C_{#3}
  \qquad \text{for all $t' \ge t_{#3}$}#4
\end{align*}}
\begin{lemma} \label{lm:uvw_bounded_hoelder}
  If Condition~\ref{cond:large_time} is fulfilled, then 
  there exist $\alpha_0 \in (0, 1)$ and $C_0 \gt 0$ such that for all $t_0 \ge 1$ we have
  \begin{align*}
    u, v \in C^{2+\alpha_0, 1+\frac{\alpha_0}{2}}(\ombar \times [t_0, t_0+1])
    \quad \text{and} \quad
    w    \in C^{\alpha_0, 1+\frac{\alpha_0}{2}}(\ombar \times [t_0, t_0+1]))
  \end{align*}
  with
  \begin{align*}
    \max\left\{
    \|u\|_{C^{2+\alpha_0, 1+\frac{\alpha_0}{2}}(\ombar \times [t_0, t_0+1])}, \,
    \|v\|_{C^{2+\alpha_0, 1+\frac{\alpha_0}{2}}(\ombar \times [t_0, t_0+1])}, \,
    \|w\|_{C^{\alpha_0, 1+\frac{\alpha_0}{2}}(\ombar \times [t_0, t_0+1])}
    \right\} \le C_0.
  \end{align*}
\end{lemma}
\begin{proof}
  Lemma~\ref{lm:u_bound_infty} and Lemma~\ref{prop:global_ex_n} assert the existence of $M \gt 0$ such that
  \begin{align*}
    \max\left\{
      \|u(\cdot, t)\|_{L^\infty(\Omega)}, \,
      \|v(\cdot, t)\|_{W^{1, \infty}(\Omega)}
    \right\} \le M
  \end{align*}
  for all $t \ge 0$.
  Therefore, the statement is mainly a consequence of known parabolic regularity theory (and Lemma~\ref{lm:bound_u_bound_w}).
  Nonetheless, we choose to include a short proof here.
  For this purpose we at first fix $0 \lt t_1 \lt t_2 \lt t_3 \lt t_4 \lt t_5 \lt 1$.

  As \cite[Theorem~1.3]{PorzioVespriHolderEstimatesLocal1993} warrants
  that there exist $\alpha_1 \in (0, 1)$ and $C_1 \gt 0$ such that \bdd{u}{\alpha}{1},
  by Lemma~\ref{lm:bound_u_bound_w} there exists $C_2 \gt 0$ with
  \begin{align*}
    w(\cdot, t') \in C^{\alpha_2}(\ombar)
    \quad \text{and} \quad
    \|w(\cdot, t')\|_{C^{\alpha_2}(\ombar)} \le C_2
    \qquad \text{for all $t' \ge t_2$},
  \end{align*}
  where $\alpha_2 \defs \alpha_1$.

  Then \cite[Theorem~IV.5.3]{LadyzenskajaEtAlLinearQuasilinearEquations1998}
  implies the existence of $\alpha_3 \in (0, 1)$ and $C_3 \gt 0$ such that \bdd{v}{2+\alpha}{3}.

  This in turn allows us to employ first \cite[Theorem~1.1]{LiebermanHolderContinuityGradient1987}
  and then again \cite[Theorem~IV.5.3]{LadyzenskajaEtAlLinearQuasilinearEquations1998}
  to obtain $\alpha_4, \alpha_5 \in (0, 1)$ and $C_4, C_5 \gt 0$ such that
  \bdd{u}{1+\alpha}{4}{} and \bdd{u}{2+\alpha}{5}.

  Finally, the asserted regularity of $w$ follows from \eqref{eq:tilde_repr_w} and the third line in \eqref{prob:p}.
\end{proof}

\subsection{Convergence of $u$ and $w$}
Again, we start by stating a rather weak convergence result:
\begin{lemma} \label{lm:very_weak_conv_u}
  If $\tmax = \infty$, then
  \begin{align} \label{eq:very_weak_conv_u:nabla_u}
    \int_1^\infty \intom \frac{|\nabla u|^2}{u^2} \lt \infty.
  \end{align}
\end{lemma}
\begin{proof}
  By multiplying the second equation in \eqref{prob:p} by $v$
  and integrating over $(0, T) \times \Omega$ (for $T \gt 0$) we obtain
  \begin{align*}
      \frac12 \intom v^2(\cdot, T) - \frac12 \intom v_0^2
    = \frac12 \ddt \int_0^T \intom v^2
    = - \int_0^T \intom |\nabla v|^2 - \int_0^T \intom v^2 w,
  \end{align*}
  hence (as $v, w \ge 0$ by Lemma~\ref{lm:positivity_mass_const})
  \begin{align} \label{eq:very:weak_conv_u:nabla_v}
    \int_0^T \intom |\nabla v|^2 \le \frac12 \intom v_0^2 \quad \text{for all } T \gt 0.
  \end{align}

  Furthermore, we have by the first equation in \eqref{prob:p}, by integrating by parts and using Young's inequality
  \begin{align*}
          -\frac{\diff}{\dt} \intom \log u 
    &=    -\intom \frac{|\nabla u|^2}{u^2} + \intom \frac{\nabla u \cdot \nabla v}{u} 
     \le  -\frac12 \intom \frac{|\nabla u|^2}{u^2} + \frac12 \intom |\nabla v|^2
  \end{align*}
  in $(0, \infty)$.
  After integration over $(1, T)$ for $T \gt 1$ this yields
  \begin{align*}
          \int_1^T \intom \frac{|\nabla u|^2}{u^2}
    &\le  2\intom \log u(\cdot, T) - 2\intom \log u(\cdot, 1) + \int_1^T \intom |\nabla v|^2 \\
    &\le  2m - 2 \log \inf_{x \in \Omega} u(x, 1) \cdot |\Omega| + \frac12 \intom v_0^2
  \end{align*}
  by Lemma~\ref{lm:positivity_mass_const} (note that $\log s \le s$ for $s \gt 0$),
  as $-\log$ is decreasing, $u(\cdot, 1) \gt 0$ in $\ombar$ by Lemma~\ref{lm:positivity_mass_const}
  and due to \eqref{eq:very:weak_conv_u:nabla_v}.
  An immediate consequence thereof is \eqref{eq:very_weak_conv_u:nabla_u}.
\end{proof}

\begin{lemma} \label{lm:convergence_u}
  If Condition~\ref{cond:large_time} is fulfilled,
  then for all $\alpha \in (0, \alpha_0)$ with $\alpha_0$ as in Lemma~\ref{lm:uvw_bounded_hoelder}
  \begin{align*}
    u(\cdot, t) \ra \ol u_0 = \frac{1}{|\Omega|} \intom u_0 \quad \text{in $C^{2+\alpha}(\ombar)$ for $1 \lt t \ra \infty$}
  \end{align*}
  holds.
\end{lemma}
\begin{proof}
  Suppose there are $\eps \gt 0$ and a sequence $(t_k)_{k \in \N} \subset (0, \infty)$ with $t_k \ra \infty$
  and
  \begin{align} \label{eq:convergence_u:contradiction_assumption}
    \left\|u(\cdot, t_k) - \ol u_0 \right\|_{C^{2+\alpha}(\ombar)} \ge \eps \quad \text{for $k \in \N$}.
  \end{align}
  As $C^{2 + \alpha_0}(\ombar) \embed \embed C^{2+\alpha}(\ombar)$,
  Lemma~\ref{lm:uvw_bounded_hoelder} allows us to find a subsequence $(t_{k_j})_{j \in \N}$ of $(t_k)_{k \in \N}$
  along which
  \begin{align*}
    u(\cdot, t_{k_j}) \ra u_\infty \quad \text{in $C^{2+\alpha}(\ombar)$ as $j \ra \infty$}
  \end{align*}
  for some $u_\infty \in C^{2+\alpha}(\ombar)$.

  Hölder's and Poincar\'e's inequalities as well as \eqref{eq:very_weak_conv_u:nabla_u} imply
  \begin{align*}
          \int_1^\infty \intom |u - \ol u_0|
    &\le  |\Omega|^\frac12 \int_1^\infty \intom (u - \ol u_0)^2 \\
    &\le  C \int_1^\infty \intom |\nabla u|^2 \\
    &\le  C \|u\|_{L^\infty(\Omega \times (0, \infty))}^2 \int_1^\infty \intom \frac{|\nabla u|^2}{u^2}
    \lt   \infty
  \end{align*}
  for some $C \gt 0$.
  As
  \begin{align*}
    \sup_{T \ge 1} \|u\|_{C^{\alpha, \frac{\alpha}{2}}(\ombar \times [T, T+1])} \lt \infty
  \end{align*}
  by Lemma~\ref{lm:uvw_bounded_hoelder} for all $T \ge 1$,
  the map $[1, \infty) \ni t \mapsto \intom |u(\cdot, t) - \ol u_0|$ is uniformly continuous.
  However, this implies $u_\infty = \ol u_0$, which contradicts~\eqref{eq:convergence_u:contradiction_assumption}.
\end{proof}

\begin{lemma} \label{lm:convergence_w}
  If Condition~\ref{cond:large_time} is fulfilled, then $w(\cdot, t) \ra \frac{\ol u_0}{\delta}$ in $C^0(\ombar)$ for $t \ra \infty$.
\end{lemma}
\begin{proof}
  Let $\eps \gt 0$.
  According to Lemma~\ref{lm:convergence_u} we may choose $t_1 \gt 0$ such that
  \begin{align*}
    \|u(\cdot, t) - \ol u_0\|_{C^0(\ombar)} \lt \frac{\eps \delta}{3} \quad \text{for all $t \gt t_1$}.
  \end{align*}
  Furthermore, there are $t_2, t_3 \gt 0$ such that
  \begin{align*}
    \ure^{-\delta t} \|w_0\|_{C^0(\ombar)} \lt \frac{\eps}{3} \quad \text{for all $t \gt t_2$}
  \end{align*}
  and
  \begin{align*}
    \frac{\|u - \ol u_0\|_{C^0(\ombar \times (0, t_1))}}{\delta} \ure^{-\delta (t - t_1)} \lt \frac{\eps}{3} \quad \text{for all $t \gt t_3$}.
  \end{align*}

  Let
  \begin{align*}
    \tilde w: \ombar \times [0, \infty) \ra \R, \quad
    (x, t) \mapsto w(x, t) - \frac{\ol u_0}{\delta} [1 - \ure^{-\delta t}],
  \end{align*}
  then we have for $t \gt t_0 \defs \max\{t_1, t_2, t_3\}$ by the representation formula \eqref{eq:tilde_repr_w}
  \begin{align*}
          \left\| \tilde w(\cdot, t) \right\|_{C^0(\ombar)}
    &=    \left\| w(\cdot, t) - \int_0^t \ure^{-\delta(t-s)} \ol u_0 \ds \right\|_{C^0(\ombar)} \\
    &\le  \ure^{-\delta t} \|w_0\|_{C^0(\ombar)} + \int_0^t \ure^{-\delta (t-s)} \|u(\cdot, s) - \ol u_0\|_{C^0(\ombar)} \ds \\
    &\lt  \frac{\eps}{3}
        + \|u - \ol u_0\|_{C^0(\ombar \times (0, t_1))} \int_0^{t_1} \ure^{-\delta (t-s)} \ds
        + \frac{\eps \delta}{3} \int_{t_1}^t \ure^{-\delta (t-s)} \ds \\
    &=    \frac{\eps}{3}
        + \|u - \ol u_0\|_{C^0(\ombar \times (0, t_1))} \frac{1}{\delta} [\ure^{-\delta (t-t_1)} - \ure^{-\delta t}]
        + \frac{\eps \delta}{3 \delta} [1 - \ure^{-\delta(t-t_1)}] \\
    &\lt  \frac{\eps}{3} + \frac{\eps}{3} + \frac{\eps}{3}
    =     \eps.
  \end{align*}

  As $\eps \gt 0$ was arbitrary, we conclude
  \begin{align*}
    \lim_{t \ra \infty} \left\| \tilde w(\cdot, t) \right\|_{C^0(\ombar)} = 0
  \end{align*}
  and therefore
  \begin{align*}
          \lim_{t \ra \infty} \left\| w(\cdot, t) - \frac{\ol u_0}{\delta} \right\|_{C^0(\ombar)}
    &\le  \lim_{t \ra \infty} \left\| \tilde w(\cdot, t) \right\|_{C^0(\ombar)}
        + \lim_{t \ra \infty} \left\| \frac{\ol u_0}{\delta} \ure^{-\delta t} \right\|_{C^0(\ombar)} \\
    &=    0 + 0 = 0.
    \qedhere
  \end{align*}
\end{proof}

\subsection{Improving the type of convergence. Proof of Theorem~\ref{th:conv}}

\begin{prop} \label{prop:convergence}
  Suppose Condition~\ref{cond:large_time} is fulfilled.
  Let $\alpha_0$ be as in Lemma~\ref{lm:uvw_bounded_hoelder} and $\alpha \in (0, \alpha_0)$,
  then \eqref{eq:main:convergence} holds.
\end{prop}
\begin{proof}
  The statements for $u$ have been been shown in Lemma~\ref{lm:convergence_u}. 

  Suppose there exist $\eps \gt 0$ and a sequence $(t_k)_{k \in \N}$ with $t_k \ra \infty$ and
  \begin{align} \label{eq:convergence:v}
    \|v(\cdot, t_k) - 0\|_{C^{2+\alpha}(\ombar)} \ge \eps \quad \text{for $k \in \N$.}
  \end{align}
  By Lemma~\ref{lm:uvw_bounded_hoelder} we could then choose a subsequence $(t_{k_j})_{j \in \N}$ of $(t_k)_{k \in \N}$ along which
  \begin{align*}
    v(\cdot, t_{k_j}) \ra v_\infty \quad \text{in $C^{2+\alpha}(\ombar)$ as $j \ra \infty$}
  \end{align*}
  for some $v_\infty \in C^{2+\alpha}(\ombar)$.
  However, Lemma~\ref{lm:convergence_v} implies $v_\infty \equiv 0$, which contradicts \eqref{eq:convergence:v}.

  The statement for $w$ can be shown analogously.
\end{proof}
Finally we are able to prove Theorem~\ref{th:conv}:
\begin{proof}[Proof of Theorem~\ref{th:conv}]
  Condition~\ref{cond:large_time} is fulfilled by Lemma~\ref{prop:cond_fulfilled}.

  Let $\alpha_0$ be as in Lemma~\ref{lm:uvw_bounded_hoelder} and $\alpha \in (0, \alpha_0)$.
  As
  \begin{align*}
    \bigcup_{T \ge 1} \ombar \times [T, T+1] = \ombar \times [1, \infty)
  \end{align*}
  and $\ombar \times [T, T+1]$ is compact for all $T \ge 1$,
  \eqref{eq:main:function_spaces} is satisfied,
  while \eqref{eq:main:convergence} has been shown in Proposition~\ref{prop:convergence}.
\end{proof}

\section*{Acknowledgment}
\small This work is based on a master thesis, which the author submitted at Paderborn University in February 2018.


\end{document}